\documentclass[10pt]{article}
\usepackage[a4paper,margin=23mm]{geometry}
\usepackage{amsmath,amssymb,amsthm}
\usepackage{microtype}
\usepackage[hidelinks]{hyperref}

\newtheorem{theorem}{Theorem}

\title{The Exact Dihedral Angle Sum Range for\\Circumcentral-Cevian Jaw Tetrahedra}
\author{Sergey Korotov\\[-1mm]
\small Department of Business and Mathematics, M\"alardalen University\\[-1mm]
\small Box 883, SE-721 23 V\"aster\aa s, Sweden\\[-1mm]
\small \texttt{sergey.korotov@mdu.se}}
\date{24 October 2026}

\begin{document}
\maketitle
\vspace{-1.3em}

\begin{abstract}
Hajja, Hammoudeh, Hayajneh and Martini showed that a tetrahedron whose cevians to the circumcenters of the opposite faces are concurrent is either a kite or a special jaw. We determine the exact dihedral-angle-sum range on the jaw component. If $\Sigma(T)$ denotes the sum of the six interior dihedral angles of such a jaw $T$, then
\[
 6\arccos\frac13\le \Sigma(T)<2\pi+\arccos\frac14.
\]
The lower endpoint is attained only by the regular tetrahedron, while the upper endpoint is approached only in the degenerate limit. The proof reduces the family to one scale-invariant parameter and an explicit monotonicity argument.
\end{abstract}


\section*{1. The result}
For a nondegenerate Euclidean tetrahedron $T$, let $\Sigma(T)$ be the sum of its six interior dihedral angles. The universal sharp bounds are $2\pi<\Sigma(T)<3\pi$; see Gaddum~\cite{Gaddum1952}, Jarden~\cite{Jarden1941}, and Katsuura~\cite{Katsuura2020}.

A \emph{jaw} is a tetrahedron with two opposite distinguished edges of lengths $x$ and $y$ and with the remaining four edges of one common length $a$. Hajja--Hammoudeh--Hayajneh--Martini~\cite[Theorem~3.9]{HajjaEtAl2020} proved that the cevians joining the vertices of a tetrahedron to the circumcenters of the opposite faces are concurrent if and only if the tetrahedron is either a kite or a jaw satisfying
\begin{equation}
 x^2y^2-2a^2(x^2+y^2)+3a^4=0.
 \label{eq:cevian}
\end{equation}
We call the latter tetrahedra \emph{circumcentral-cevian jaws}.

For tetrahedral kites, the exact dihedral-angle-sum range
\[
 \left[6\arccos\frac13,\,3\pi\right)
\]
is known; see Korotov~\cite{Korotov2026}. For related background on tetrahedral kites, see also Hajja--Hammoudeh--Hayajneh~\cite{HajjaHammoudehHayajneh2015}. We therefore restrict attention to the jaw component.

The kite and jaw families intersect only in the regular tetrahedra. Indeed, in this case $x=y=a$, and \eqref{eq:cevian} is immediately satisfied:
\[
 a^4-2a^2(2a^2)+3a^4=0.
\]

\begin{theorem}\label{thm:main}
For the class $\mathcal J_c$ of nondegenerate circumcentral-cevian jaws,
\[
 \{\Sigma(T):T\in\mathcal J_c\}
 =\left[6\arccos\frac13,\,2\pi+\arccos\frac14\right).
\]
The lower endpoint is attained only by the regular tetrahedron. Numerically, the interval is
\[
 [423.172676\ldots^\circ,\,435.522488\ldots^\circ).
\]
\end{theorem}

\begin{proof}
By scaling, take $a=1$ and put
\[
 u=x^2>0,\qquad v=y^2>0.
\]
Then \eqref{eq:cevian} becomes
\begin{equation}
 uv-2(u+v)+3=0.
 \label{eq:uv}
\end{equation}
Because the four cross-edges all have length $1$, the endpoints of each distinguished edge are equidistant from the endpoints of the other. Thus both endpoints of one distinguished edge lie in the perpendicular-bisector plane of the other edge. The first edge therefore lies in that plane, while the second is normal to it; hence the two distinguished edges are orthogonal. After a rigid motion, every such jaw can therefore be written as
\[
 A_1=\left(-\frac{\sqrt u}{2},0,0\right),\quad
 A_2=\left( \frac{\sqrt u}{2},0,0\right),
\]
\[
 A_3=\left(0,c,\frac{\sqrt v}{2}\right),\quad
 A_4=\left(0,c,-\frac{\sqrt v}{2}\right),\qquad
 c=\frac12\sqrt{4-u-v}>0.
\]
Indeed, each cross-edge then has squared length $(u+v)/4+c^2=1$, and $c>0$ expresses nondegeneracy.

Set
\[
 \rho:=\sqrt{uv}>0.
\]
From \eqref{eq:uv},
\begin{equation}
 u+v=\frac{\rho^2+3}{2}.
 \label{eq:sumuv}
\end{equation}
Moreover,
\[
 d:=4-u-v=\frac{5-\rho^2}{2}>0.
\]
The arithmetic--geometric mean inequality gives
\[
 2\rho\le u+v=\frac{\rho^2+3}{2},
\]
so $(\rho-1)(\rho-3)\ge0$. Since $d>0$ implies $\rho<\sqrt5<3$, it follows that
\[
 0<\rho\le1.
\]
If $\rho=1$, equality holds in the arithmetic--geometric mean inequality, hence $u=v=1$, i.e. the regular tetrahedron. Conversely, every $\rho\in(0,1]$ occurs: $u,v$ are the positive roots of
\[
 z^2-\frac{\rho^2+3}{2}z+\rho^2=0,
\]
whose discriminant is
\[
 \frac14(1-\rho^2)(9-\rho^2)\ge0.
\]
For these roots $u+v\le2$, so $d=4-u-v\ge2>0$, and the coordinate realization above gives a nondegenerate jaw. Thus every $\rho\in(0,1]$ produces a nondegenerate circumcentral-cevian jaw, and conversely every such jaw gives a parameter $\rho\in(0,1]$. Hence $(0,1]$ parametrizes the entire family of nondegenerate circumcentral-cevian jaws, up to similarity and interchange of the two distinguished edges.

Let $\alpha$ and $\beta$ be the dihedral angles at the distinguished edges of lengths $\sqrt u$ and $\sqrt v$, respectively, and let $\gamma$ be the common dihedral angle at each of the other four edges. In the plane perpendicular to $A_1A_2$, the two face directions are represented by $(c,\sqrt v/2)$ and $(c,-\sqrt v/2)$. Hence
\[
 \cos\alpha
 =\frac{c^2-v/4}{c^2+v/4}
 =\frac{4-u-2v}{4-u}.
\]
Similarly,
\[
 \cos\beta
 =\frac{c^2-u/4}{c^2+u/4}
 =\frac{4-v-2u}{4-v}.
\]
For a cross-edge, say $A_1A_3$, take a normal to the face $A_1A_2A_3$ directed away from $A_4$ and a normal to the face $A_1A_3A_4$ directed toward $A_2$:
\[
 n_1=\left(0,-\frac{\sqrt v}{2},c\right),\qquad
 n_2=\left(c,-\frac{\sqrt u}{2},0\right).
\]
The angle between these normals is the interior dihedral angle $\gamma$, and their scalar product and norms give
\[
 \cos\gamma
 =\frac{\sqrt{uv}/4}
 {\sqrt{(c^2+v/4)(c^2+u/4)}}
 =\sqrt{\frac{uv}{(4-u)(4-v)}}.
\]
Thus
\begin{equation}
 \cos\alpha=\frac{4-u-2v}{4-u},\qquad
 \cos\beta =\frac{4-v-2u}{4-v},\qquad
 \cos\gamma=\sqrt{\frac{uv}{(4-u)(4-v)}}.
 \label{eq:angles}
\end{equation}
From \eqref{eq:angles},
\[
 \tan^2\frac\alpha2=\frac vd,\qquad
 \tan^2\frac\beta2=\frac ud,\qquad
 \cos\gamma=\frac{\rho}{\sqrt{10-\rho^2}}.
\]
If $A=\tan(\alpha/2)$ and $B=\tan(\beta/2)$, then
\[
 AB=\frac{2\rho}{5-\rho^2},\qquad
 A^2+B^2=\frac{\rho^2+3}{5-\rho^2}.
\]
Hence
\begin{equation}
 \tan^2\frac{\alpha+\beta}{2}
 =R(\rho):=
 \frac{(\rho+1)(\rho+3)(5-\rho^2)}{(5-\rho^2-2\rho)^2}.
 \label{eq:R}
\end{equation}
Moreover, $AB=2\rho/(5-\rho^2)\le 1/2<1$, so
$\tan((\alpha+\beta)/2)=(A+B)/(1-AB)>0$; hence
$0<(\alpha+\beta)/2<\pi/2$ and therefore
$(\alpha+\beta)/2=\arctan\sqrt{R(\rho)}$.
Since $u$ and $v$ are determined by $\rho$ up to interchange, and the dihedral-angle sum is symmetric in $u$ and $v$, it depends only on $\rho$. We therefore write it as $\Sigma(\rho)$. Thus
\begin{equation}
 \Sigma(\rho)=
 2\arctan\sqrt{R(\rho)}
 +4\arccos\frac{\rho}{\sqrt{10-\rho^2}},
 \qquad 0<\rho\le1.
 \label{eq:Sigma}
\end{equation}
Differentiating and simplifying gives
\begin{equation}
 \Sigma'(\rho)=
 \frac{2}{(10-\rho^2)\sqrt{5-\rho^2}}
 \left[
 \frac{20+15\rho-\rho^3}{\sqrt{(\rho+1)(\rho+3)}}-10\sqrt2
 \right].
 \label{eq:derivative}
\end{equation}
For $0\le\rho\le1$, both quantities in the brackets in \eqref{eq:derivative} are positive. Hence, to prove that the bracket is negative, it suffices to compare their squares. We have
\begin{align*}
 &200(\rho+1)(\rho+3)-(20+15\rho-\rho^3)^2\\
 &\qquad=-\rho^6+30\rho^4+40\rho^3-25\rho^2+200\rho+200.
\end{align*}
Since $0\le\rho\le1$, the terms $30\rho^4$, $40\rho^3$, and $200\rho$ are nonnegative, while
$-\rho^6\ge-1$ and $-25\rho^2\ge-25$. Therefore
\[
 -\rho^6+30\rho^4+40\rho^3-25\rho^2+200\rho+200
 \ge 200-25-1=174>0.
\]
Therefore
\[
 \frac{20+15\rho-\rho^3}{\sqrt{(\rho+1)(\rho+3)}}<10\sqrt2,
\]
and consequently $\Sigma'(\rho)<0$ on $(0,1]$.

At $\rho=1$ we have the regular tetrahedron, so
\[
 \Sigma(1)=6\arccos\frac13.
\]
As $\rho\downarrow0$, equation \eqref{eq:Sigma} gives $\gamma\to\pi/2$ and $R(\rho)\to3/5$. Since
\[
 \cos\left(2\arctan\sqrt{\frac35}\right)=\frac14,
\]
we obtain
\[
 \lim_{\rho\downarrow0}\Sigma(\rho)
 =2\pi+\arccos\frac14.
\]
The limit is not attained because $\rho=0$ is degenerate. Since every $\rho\in(0,1]$ corresponds to a nondegenerate circumcentral-cevian jaw and $\Sigma(\rho)$ is continuous and strictly decreasing, as $\rho$ runs from $1$ down to $0$, the value $\Sigma(\rho)$ runs through every value from
\[
 6\arccos\frac13
\]
up to, but not including,
\[
 2\pi+\arccos\frac14.
\]
Thus the interval in Theorem~\ref{thm:main} is the exact range, and the lower endpoint is attained only by the regular tetrahedron.
\end{proof}

\paragraph{Remark.}
The interval in Theorem~\ref{thm:main} has width only
\[
 \arccos\frac14-\left(6\arccos\frac13-2\pi\right)
 =12.349812\ldots^\circ,
\]
so the jaw component of the circumcentral-cevian family is highly rigid with respect to the dihedral-angle sum.

\end{document}